\documentclass[12pt]{amsart}
\usepackage[left=3cm,right=3cm]{geometry}
\newtheorem{theorem}{Theorem}[section]
\newtheorem{corollary}[theorem]{Corollary}
\newtheorem{lemma}[theorem]{Lemma}

\theoremstyle{definition}
\newtheorem{definition}[theorem]{Definition}
\theoremstyle{remark}

\numberwithin{equation}{section}

\usepackage{amsmath, amsthm, amssymb}
\usepackage{url} 
\usepackage{graphicx}
\usepackage{xcolor}

\def\m{\mathbb M}

\def\r{\mathbb R}

\def\s{\mathbb S}
\def\h{\mathbb H}
\def\d{\mathbb D}

\def\c{\mathsf{C}}
\def\q{\mathbf{q}}
\def\L{\mathcal{L}}
\def\E{\mathcal{E}}

\def\ekt{\mathbb{E}(\kappa,\tau)}
\begin{document}

\title[Stability of cylinders in $\ekt$  homogeneous spaces]
{Stability of cylinders in $\ekt$ homogeneous spaces}
\author{Antonio Bueno}
\address{ Departamento de Ciencias. Centro Universitario de la Defensa de San Javier. 30729 Santiago de la Ribera, Spain}
\email{antonio.bueno@cud.upct.es}
\author[Rafael L\'opez]{Rafael L\'opez}
\address{Departamento de Geometr\'{\i}a y Topolog\'{\i}a Universidad de Granada 18071 Granada, Spain}
\email{rcamino@ugr.es}
\subjclass{Primary 53A10; Secondary 53C42}

\keywords{ $\ekt$ spaces, stability, Plateau-Rayleigh instability, Morse index, partitioning problem.}

\begin{abstract}
We extend the classical Plateau-Rayleigh instability criterion   in the $\ekt$ spaces. We prove the existence of a positive number $L_0>0$ such that if a truncated  circular cylinder of radius $\rho$ in $\ekt$ has length $L>L_0$ then it is unstable. This number $L_0$ depends on $\kappa$, $\tau$ and $\rho$.  The value $L_0$ is sharp under axially-symmetric variations of the surface. We also extend this result for the partitioning problem in $\ekt$.
\end{abstract}

\maketitle

\section{Introduction and statement of results}\label{s1}

In capillary theory \cite{de}, the Plateau-Rayleigh instability criterion   asserts that a truncated piece of a circular cylinder of radius $\rho>0$ in $\r^3$ is unstable if its length $L$ satisfies 
\begin{equation}\label{pla}
L>2\pi \rho.
\end{equation}
Circular cylinders are surfaces with constant mean curvature (cmc to abbreviate) and the inequality \eqref{pla} can be derived from a well-established theory of stability of cmc surfaces going back, at least, to the initial works of Barbosa, do Carmo and Eschenburg  \cite{bc,bce}. An analogous bound for cmc cylinders in the hyperbolic 3-space has been recently exhibited by the authors \cite{bulo}, and similar instability criteria for cylindrical liquids  have been obtained in other contexts of the capillary theory: see e.g. \cite{ag,blo,lo0,lo3,mc} and references therein.

In this paper, we consider the stability of truncated cylinders in  the family of  simply-connected homogeneous $3$-dimensional spaces whose isometry group has dimension $4$. These spaces can be parametrized by two real parameters $\kappa,\tau$ and are known as the $\ekt$ spaces. The $\ekt$ spaces complete the classification of the Thurston geometries along with  the space forms $\r^3$, $\h^3$ and $\s^3$, whose isometry group are of dimension $6$, and the space $Sol$, whose isometry group is of dimension $3$. The $\ekt$ spaces admit a Riemannian submersion $\pi:\ekt\to\m^2(\kappa)$ onto the $2$-dimensional space form $\m^2(\kappa)$  and their bundle curvature is $\tau$. If $\tau=0$, then $\mathbb{E}(\kappa,0)$ is one the  product spaces $\m^2(\kappa)\times\r$.  If $\tau\neq0$, then $\ekt$ is  the Heisenberg space   if $\kappa=0$; the universal cover of the special linear group   if $\kappa<0$; and the Berger spheres   if $\kappa>0$.

A local model for the $\ekt$ spaces is the following. If $r>0$, let $\d(r)=\{(x,y)\in\r^2\colon x^2+y^2<r^2\}$ be the disk of radius $r>0$.   
Let $\mathcal{R}(\kappa,\tau)$ be the space $\r^3$ if $\kappa\geq0$ or   $\mathbb{D}(2/\sqrt{-\kappa})\times\r$ if $\kappa<0$. Let us endow $\mathcal{R}(\kappa,\tau)$ with coordinates $(x,y,z)$ and the metric
$$
g=\sigma^2(dx^2+dy^2)+(\sigma\tau(ydx-xdy)+dz)^2,\quad \sigma=\frac{4}{4+\kappa(x^2+y^2)}.
$$
Then, $(\mathcal{R}(\kappa,\tau),g)$ is isometric to the   $\ekt$ space. The Riemannian submersion is isomorphic to the projection onto the first two coordinates, $\pi\colon \mathcal{R}(\kappa,\tau)\to \r^2$ if $\kappa\geq 0$ and $\pi\colon \mathcal{R}(\kappa,\tau)\to \mathbb{D}(2/\sqrt{-\kappa})$ if $\kappa< 0$. Notice that the base space in this submersion is the two-dimensional space form $\m^2(\kappa)$, equipped with the metric $\sigma^2(dx^2+dy^2)$. When $\kappa\leq0$ this model for the $\ekt$ spaces is global but if  $\kappa>0$, then  this model omits one fiber. Indeed, in the case $\kappa>0,\tau>0$ of the Berger spheres $\s^3_b$, if we regard $\s^3_b=\{(z,w)\in\mathbb{C}^2\colon |z|^2+|w|^2=1\}$ then an explicit isometry between $\mathcal{R}(\kappa,\tau)$ and $\s^3_b-\{(e^{i\theta},0)\colon\theta\in\r\}$ is
$$
\Psi(x,y,z)=\frac{1}{\sqrt{1+\frac{\kappa}{4}(x^2+y^2)}}\left(\frac{\sqrt{\kappa}}{2}(x+iy)e^{i\frac{\kappa}{4\tau}z},e^{i\frac{\kappa}{4\tau}z}\right).
$$
With this isometry we see that two points $(x,y,z)$ and $(x,y,z+\frac{8\tau\pi}{\kappa})$ are identified to the same point in $\s^3_b$. Finally, if $\kappa=\tau=0$ in this model, then $\mathbb{E}(0,0)$ is simply the Euclidean space $\r^3$, whose isometry group is of dimension $6$.

In the last decades, the theory of cmc surfaces in the $\ekt$ spaces has received the attention of many researchers since the extension of Hopf's theorem  by Abresch and Rosenberg \cite{abro}. This produced a vast literature  which, without aiming to collect it, we refer the reader to \cite{dan,dhm,femi} and references therein. One the most relevant topics   in the $\ekt$ spaces is the study of the stability of cmc surfaces  \cite{lm,mpr,mmp,ro}.   A cmc surface is said to be \emph{stable} if it is  a second order minimizer of the area functional under the preservation of the volume. In case  that we drop the volume preserving condition, the surface  is said \emph{strongly stable}. 

Our aim in this paper is the extension of the  Plateau-Rayleigh estimate \eqref{pla} in the context of   $\ekt$ spaces, and for that matter we need to generalize the analogues of the circular cylinders. The natural notion of a cylinder is the lifting in $\ekt$ of a circle of $\m^2(\kappa)$, where by a circle we mean a closed curve of $\m^2(\kappa)$ with constant geodesic curvature $\kappa_g$. In particular, if $\kappa\geq0 $ then $\kappa_g$ is any positive constant and if $\kappa<0$, then   $\kappa_g^2>-\kappa$. 

\begin{definition}   Given a curve   $\alpha\colon I\subset\r\to\m^2(\kappa)$,   the  {\it vertical cylinder}   over $\alpha$ is defined as  $\c_\alpha=\pi^{-1}(\alpha(I))$. The vertical cylinder is called {\it circular} if  $\alpha$ is a circle. In such a case, the radius of $\c_\alpha$ is  the radius of  $\alpha$.
\end{definition}

Notice that a vertical cylinder   $\c_\alpha$ has constant mean curvature if and only if $\kappa_g$ is constant because the mean curvature of $\c_\alpha$ is $\kappa_g/2$. It was proved in  \cite{mpr} that a cmc vertical cylinder $\c_\alpha$ is strongly stable if and only if   $\kappa_g^2+\kappa\leq0$. In particular this implies  $\kappa\leq 0$. If   $\kappa=0$, then  $\alpha$ is a geodesic ($\kappa_g=0$), but if  $\kappa<0$ there are three possibilities:  $\alpha$ is  a horocycle ($\kappa_g^2=-\kappa$),  an equidistant curve ($0<\kappa_g^2<-\kappa$) or a geodesic ($\kappa_g=0$). 

From the result of \cite{mpr},  the only cmc vertical cylinders that are not strongly stable are the circular ones. Following with  the   spirit of the classical Plateau-Rayleigh instability criterion, we want to establish conditions of stability of compact pieces of circular cylinders in the $\ekt$ spaces in terms of their length. We precise the terminology. Given a circular cylinder   $\c_\alpha$ and $L>0$, we define the \emph{truncated circular cylinder of length $L$} as the set $\c_\alpha(L)=\{(x,y,z)\in\c_\alpha\colon 0\leq z\leq L\}$. In other words,  $\c_\alpha(L)$ is the compact piece of  $\c_\alpha$ between the planes of equations $z=0$ and $z=L$. Since the translations $(x,y,z)\mapsto (x,y,z+t)$ are isometries in $\ekt$, the fact to fix   the $z$-coordinate between $0$ and $L$ does not lose generality in this definition.

The first result in this paper is the extension of the classical Plateau-Rayleigh result  in the $\ekt$ spaces.

\begin{theorem}\label{t1}
Let $\c_\alpha$ be a  circular cylinder of radius $\rho$  in $\ekt$. If $L_0$ is given by
\begin{equation}\label{pla1}
L_0=\left\{\begin{array}{ll}
\frac{2\pi}{-\kappa}\sinh(\rho\sqrt{-\kappa})\sqrt{-\kappa+4\tau^2\tanh^2(\frac{\rho\sqrt{-\kappa}}{2})}&\kappa<0,\\
2\pi \rho\sqrt{1+\tau^2\rho^2}&\kappa=0,\\
\frac{2\pi}{\kappa}\sin(\rho\sqrt{\kappa})\sqrt{\kappa+4\tau^2\tan^2(\frac{\rho\sqrt{\kappa}}{2})}&\kappa>0,
\end{array}
\right.
\end{equation}
then if $L>L_0$ the truncated cylinder $\c_\alpha(L)$ is not stable. If $\kappa>0$ and $\tau\neq0$ we assume, in addition, $L<8\tau\pi/\kappa$.
\end{theorem}

Notice that if $\kappa=\tau=0$, then $L_0=2\pi\rho$ and theorem rediscovers the classical Plateau-Rayleigh instability criterion \eqref{pla} in $\r^3$. In the case $\kappa>0$ and $\tau\neq 0$ corresponding to the Berger spheres, the assumption $L<8\tau\pi/\kappa$ comes from the fact that if $L\geq 8\tau\pi/\kappa$ then $\c_\alpha(L)$ is identified as a torus by the periodicity of the fibers. We will also prove that this   inequality is optimal in case that we consider axially-symmetric variations of the surface (Cor. \ref{cor}).

In the second result of this paper we study the stability of vertical cmc cylinders when regarded as solutions of the partitioning problem in the $\ekt$ spaces. The general setting   is analogous to the Euclidean case and $3$-manifolds in general \cite{nit,rv,so}. Given a domain $W\subset\ekt$ with smooth boundary $\partial W$, a surface $\Sigma$ with $\mathrm{int}(\Sigma)\subset\mathrm{int}(W)$ and $\partial\Sigma\subset\partial W$ is a \emph{capillary surface} if is a critical point of the area functional among all surfaces in these conditions that separate $W$ into two domains of prescribed volumes. Any capillary surface is characterized by the fact that has constant mean curvature and the contact angle that makes $\Sigma$ with $\partial W$ along $\partial\Sigma$ is constant. In this context,  we have   similar notions of stability.

In the $\ekt$ spaces, we investigate the stability in the partitioning problem of truncated pieces of vertical cmc cylinders between two planes. To be precise, let $\Pi_c=\{(x,y,z)\in\ekt\colon z=c\}$, $c\in\r$. Let $\c_\alpha$ be a vertical cmc cylinder and $\c_\alpha(L)=\{(x,y,z)\in\c_\alpha:0\leq z\leq L\}$. Notice that we are now including the case that $\kappa_g^2+\kappa\leq 0$, that is, $\alpha$ is not circle. If $\alpha$ is a circle, then   the intersection of $\c_\alpha(L)$ with  the  support planes $\Pi_0\cup\Pi_L$ is orthogonal, hence $\c_\alpha(L)$ is a capillary surface. The result that we prove is the following. 

\begin{theorem}\label{t2}
Let $\c_\alpha(L)$ be a truncated circular cylinder of radius $\rho$ and length $L$ in $\ekt$, supported on the planes $\Pi_0\cup\Pi_L$. If $L>L_0$, where $L_0$ is given by \eqref{pla1} then $\c_\alpha(L)$ is not stable in the partitioning problem. If $\kappa,\tau>0$, we assume moreover $L<8\tau\pi/\kappa$.
\end{theorem}

In case that $\kappa_g^2+\kappa\leq 0$, the vertical cmc cylinder $\c_\alpha(L)$ makes constant contact angle with $\Pi_0\cup\Pi_L$  only if the ambient space is $\h^2(\kappa)\times\r$. In such a case, we will prove in Thm. \ref{t5} that $\c_\alpha(L)$ is strongly stable regardless of the value $L$. Theorem \ref{t5} is analogous to the result proved in \cite{mpr} but in the context of the partitioning problem.  

The  organization of the paper is the following. In Sect. \ref{s2} we investigate the model $\mathcal{R}(\kappa,\tau)$ of the $\ekt$ spaces. In Sect. \ref{s3} we introduce the two variational problems considered in this paper. In both cases,   a self-adjoint elliptic operator is defined and the stability problem can be reformulated in terms of the eigenvalues of this operator under suitable boundary conditions. For  Thm. \ref{t1} the eigenvalue problem has Dirichlet conditions, while in the case of Thm. \ref{t2} are of Neumann type. The proof of Thm. \ref{t1} is given in Sect. \ref{s4} and the proof of Thm. \ref{t2} is exhibited in Sect. \ref{s5}.

\section{Vertical cylinders in $\ekt$}\label{s2}

For each $\ekt$ space, we consider the model $\mathcal{R}(\kappa,\tau)$ described in Sect. \ref{s1}. In this model,  the vector field $E_3=\partial_z$ is a Killing vector field and $\tau$ is characterized by the property $\overline{\nabla}_XE_3=\tau X\times E_3$ for all $X\in\mathfrak{X}(\ekt)$, where $\overline{\nabla}$ is the Levi-Civita connection of $\ekt$.  Using this  model, an orthonormal basis of the tangent space is given by $\{E_1,E_2,E_3\}$, where
\begin{equation*}
\begin{split}
E_1=\frac{1}{\sigma}\partial_x-\tau  y\partial_z,\quad
E_2=\frac{1}{\sigma}\partial_y+\tau  x\partial_z,\quad
E_3=\partial_z.
\end{split}
\end{equation*}
Hence, we have 
\begin{equation}\label{equi}
\begin{split}
\partial_x=\sigma (E_1+\tau  y E_3),\quad
\partial_y=\sigma (E_2-\tau  x E_3),\quad
\partial_z=E_3.
\end{split}
\end{equation}
The Levi-Civita connection $\overline{\nabla}$ is determined by the relations
\begin{equation}\label{eqc}
\begin{array}{lll}
\overline{\nabla}_{E_1}E_1=-\dfrac{\sigma_y}{\sigma^2}E_2, &\overline{\nabla}_{E_1}E_2=\dfrac{\sigma_y}{\sigma^2}E_1+\tau E_3, & \overline{\nabla}_{E_1}E_3=-\tau E_2 ,\\
\overline{\nabla}_{E_2}E_1=\dfrac{\sigma_x}{\sigma^2}E_2-\tau E_3,& \overline{\nabla}_{E_2}E_2=-\dfrac{\sigma_x}{\sigma^2}E_1, & \overline{\nabla}_{E_2}E_3=\tau E_1,\\
\overline{\nabla}_{E_3}E_1=-\tau E_2, &\overline{\nabla}_{E_3}E_2=\tau E_1, &\overline{\nabla}_{E_3}E_3=0,\\
\end{array}
\end{equation}
Here, $\sigma_x$ and $\sigma_y$ stand for the partial derivatives of $\sigma$ with respect to $x$ and $y$, respectively.

Let $\c_\alpha$ be a vertical cylinder, where  $\alpha\colon I\to \m^2(\kappa)$ is a regular curve. If  $\alpha$ is parametrized by arc-length, a basis of the tangent plane on $\c_\alpha$ is $\{\alpha',E_3\}$ where $\alpha'$ refers to the horizontal lift of $\alpha'$   through the submersion $\pi$. Let $N$ be the unit normal vector on $\c_\alpha$ chosen so $\alpha'\times E_3=N$. The matricial expression of the second fundamental form $A$ of $\c_\alpha$   with respect to $\{\alpha',E_3\}$ is
$$A=
\left(
\begin{array}{ll}
\langle \overline{\nabla}_{\alpha'}\alpha',N\rangle & \langle\overline{\nabla}_{\alpha'}E_3,N\rangle\\
\langle\overline{\nabla}_{E_3}\alpha',N\rangle & \langle\overline{\nabla}_{E_3}E_3,N\rangle
\end{array}
\right)=\left(
\begin{array}{ll}
\kappa_g & \tau\\
\tau & 0
\end{array}
\right),
$$
where $\kappa_g$ is the geodesic curvature of $\alpha$ as a curve in $\m^2(\kappa)$.  The following result compiles some properties of the vertical cylinders: see \cite[Appendix]{mpr}.

\begin{lemma}\label{lem1}
Let $\c_\alpha$ be a vertical cylinder.   Then, 
\begin{enumerate}
\item The Gauss curvature of $\c_\alpha$ is $K=0$ and the mean curvature   is $H=\kappa_g/2$.
\item The norm of the second fundamental form $A$ is $|A|^2=\kappa_g^2+2\tau^2$.
\item The Ricci curvature along the direction of $N$ is $\mathrm{Ric}(N)=\kappa-2\tau^2$.
\end{enumerate}
\end{lemma}

Since  the metric on $\m^2(\kappa)$ is $\sigma^2(dx^2+dy^2)$, the circle $\alpha$  can be parametrized, up to an isometry in $\m^2(\kappa)$, by  $$
\alpha(s)=\left(r\cos\frac{s}{\sigma r},r\sin\frac{s}{\sigma r}\right),\qquad r>0,\,  \sigma=\frac{4}{4+\kappa r^2}.
$$
The center of $\alpha$ is the origin $(0,0)$  of $\r^2$ if $\kappa\geq0$ or $\mathbb{D}(2/\sqrt{-\kappa})$ if $\kappa<0$. The radius $\rho$ of $\alpha$ is given by 
$$
\rho=\left\{\begin{array}{ll}
\frac{2}{\sqrt{-\kappa}}\mathrm{arctanh}\frac{r\sqrt{-\kappa}}{2}&\kappa<0,\\
r&\kappa=0,\\
\frac{2}{\sqrt{\kappa}}\arctan\frac{r\sqrt{\kappa}}{2}&\kappa>0.
\end{array}\right.$$
The geodesic curvature $\kappa_g$ of $\alpha$ in $\m^2(\kappa)$ is constant and  given by
\begin{equation}\label{eqkg}
\kappa_g=\frac{4-\kappa r^2}{4r}=\left\{\begin{array}{ll}
\sqrt{-\kappa}\coth(\rho\sqrt{-\kappa})&\kappa<0,\\
\dfrac{1}{\rho}&\kappa=0,\\
\sqrt{\kappa}\cot(\rho\sqrt{\kappa})&\kappa>0.
\end{array}\right.
\end{equation}
Let
$$
R= r\sigma=\frac{4r}{ 4+\kappa r^2 }=
\left\{\begin{array}{ll}
\frac{1}{\sqrt{-\kappa}}\sinh(\rho\sqrt{-\kappa})&\kappa<0,\\
\rho&\kappa=0,\\
\frac{1}{\sqrt{\kappa}}\sin(\rho\sqrt{\kappa})&\kappa>0.
\end{array}\right.
$$
The lifting of $\alpha$ via the submersion provides a parametrization of   $\c_\alpha$, namely, 
\begin{equation}\label{para}
\psi(s,t)=\left(r\cos\frac{s}{R},r\sin\frac{s}{R},t\right),\quad s,t\in\r.
\end{equation}
We compute the first fundamental form $(g_{ij})$. Using \eqref{equi}, we obtain
\begin{equation*}
\begin{split}
\psi_s&=\frac{1}{\sigma}(-\sin\frac{s}{R}, \cos\frac{s}{R},0)=-\sin \frac{s}{R} E_1+\cos \frac{s}{R} E_2-\tau r E_3,\\
\psi_t&=(0,0,1)=E_3.
\end{split}
\end{equation*}
Thus $(g_{ij})$ and its inverse $(g^{ij})$ are the matrices
\begin{equation}\label{eqmetrica}
(g_{ij})=\left(\begin{matrix}
1+r^2\tau^2&-r\tau\\
-r\tau&1
\end{matrix}\right),\quad  (g^{ij}) =\left(\begin{matrix}
1&r\tau\\
r\tau&1+r^2\tau^2
\end{matrix}\right).
\end{equation}
Finally, by Lem. \ref{lem1}, the mean curvature $H$ of $\c_\alpha$ is constant with $H=\kappa_g/2$.

\section{Stability of cmc surfaces in $\ekt$ spaces}\label{s3}

In this section, we recall the notions of stability and index of a cmc surface in the contexts of fixed boundary (Thm. \ref{t1}) and the partitioning problem (Thm. \ref{t2}). We begin with the first case. Let $\Sigma$ be an oriented surface in $\ekt$ and let $\{\Sigma_t:t\in (-\epsilon,\epsilon)\}$ be a compactly supported variation of $\Sigma$. If we define   the functionals
$$
\mathcal{A}(t)=\mbox{Area}(\Sigma_t),\qquad \mathcal{V}(t)=  \mbox{Volume}(\Sigma_t),
$$ 
 it is known that $\Sigma$ is a critical point of $\mathcal{A}$ for all volume preserving variations if and only the mean curvature $H$ of $\Sigma$ is constant. In such a case, $\Sigma$ is said to be {\it stable} if   $\mathcal{A}''(0)\geq 0$ for all compactly supported normal variations that preserve the volume of $\Sigma$. If we drop the volume preserving condition in the variations of $\Sigma$, we say that $\Sigma$ is \emph{strongly stable}. Stability of $\Sigma$ is equivalent to  
\begin{equation}\label{e2}
\mathcal{A}''(0)=-\int_\Sigma u(\Delta u+|A|^2u+\mathrm{Ric}(N)u)\geq0,
\end{equation}
for all $u\in C_0^\infty(\Sigma)$ such that $\int_\Sigma u=0$.  Here $\Delta$ is the Laplacian operator on $\Sigma$, $A$ is the second fundamental form of $\Sigma$, $N$ is the unit normal vector field of $\Sigma$ and $\mathrm{Ric}$ the Ricci curvature of $\ekt$. The mean zero integral  $\int_\Sigma u=0$ comes from the condition that the variations preserve the volume of $\Sigma$. In consequence,   $\Sigma$ is strongly stable if $\mathcal{A}''(0)\geq 0$ for all $u\in C_0^\infty(\Sigma)$.

The parenthesis in \eqref{e2} defines the {\it Jacobi operator}   by 
\begin{equation}\label{jac}
\L=\Delta+|A|^2+\mathrm{Ric}(N),
\end{equation}
which is a self-adjoint elliptic operator. Since $\L$ is self-adjoint, we define  the quadratic form  $Q$ by 
$$Q[u]=-\int_\Sigma u\cdot \L[u],$$
in the space ${\mathcal V}=\{u\in C_0^\infty(\Sigma)\colon \int_\Sigma u=0\}$.  The  {\it weak Morse index} of $\Sigma$, denoted by $\mbox{index}_w(\Sigma)$, is defined as the maximum dimension of any subspace of ${\mathcal V}$ on which $Q$ is negative definite.   In a certain sense, the weak index measures the ways to reduce the area of $\Sigma$, up to second order, preserving the volume of $\Sigma$. Thus $\Sigma$ is stable if and only if $\mbox{index}_w(\Sigma)=0$.

Suppose $\Sigma$ is compact. Then $\mbox{index}_w(\Sigma)$  coincides with the number of negative eigenvalues $\lambda$ of the   eigenvalue problem     
\begin{equation}\label{eq10}
\left\{\begin{split}
\L[u]+\lambda u=0&\mbox{ in}\ \Sigma,\\
u=0& \mbox{ in}\  \partial\Sigma,\\
u&\in{\mathcal V}.
\end{split}\right.
\end{equation}
Since the condition $\int_\Sigma u=0$ is difficult to work with,  instead of \eqref{eq10}, we consider the eigenvalue problem
\begin{equation}\label{p-e}
\left\{\begin{split}
\L[u]+\lambda u=0&\mbox{ in}\ \Sigma,\\
u=0& \mbox{ in}\  \partial\Sigma,\\
u& \in C_0^\infty(\Sigma).
\end{split}\right.
\end{equation}
By the ellipticity of $\L$, it is well-know that the eigenvalues of \eqref{p-e} (also  of \eqref{eq10}) are ordered as a discrete spectrum  $\lambda_1<\lambda_2 \leq\lambda_3\cdots\nearrow \infty$ counting multiplicity.   The {\it Morse index} of $\Sigma$, denoted $\mbox{index}(\Sigma)$, is the number of negative eigenvalues of \eqref{p-e}. Since in \eqref{p-e} it is only required that $u$ belongs to  $C_0^\infty(\Sigma)$, we have that $\mbox{index}(\Sigma)=0$ if and only if $\Sigma$ is strongly stable or equivalently, $\lambda_1\geq 0$.   Both indexes are related by the inequalities   
\begin{equation}\label{ii}
\mbox{index}_w(\Sigma)\leq \mbox{index}(\Sigma)\leq \mbox{index}_w(\Sigma)+1.
\end{equation}
The instability criterion in the Plateau-Rayleigh estimate \eqref{pla} and in the results of this paper (Thms. \ref{t1} and \ref{t2}) are obtained once we impose the condition   $\lambda_2<0$ in the eigenvalue problem \eqref{p-e}. In such a case, the inequalities \eqref{ii} imply     $\mathrm{index}_w(\Sigma)\geq1$. This proves that $\Sigma$ is not stable.

When $\Sigma$ is not compact, the definition of the index of $\Sigma$ is given by taking  an exhaustion $\Sigma_1\subset\Sigma_2\subset\ldots\subset  \Sigma$ by bounded subdomains of $\Sigma$.  Then the weak Morse index and the Morse index of $\Sigma$ are defined by 
\begin{equation}\label{stable-infinito}
\mbox{index\,}_w(\Sigma)=\lim_{n\to\infty}\mbox{index\,}_w(\Sigma_n),\quad \mbox{index\,}(\Sigma)=\lim_{n\to\infty}\mbox{index\,}(\Sigma_n).
\end{equation}
These definitions  are independent of the choice of the exhaustion of $\Sigma$. Both numbers can be infinite, but if they are finite, then the relation \eqref{ii} holds too.  

Stability in the partitioning problem is defined  similarly. Let $S$ be a surface that separates $\ekt$ into two domains and name $W$ to one of them, having $\partial W=S$. Let $\Sigma$ be an orientable    surface with non-empty boundary $\partial\Sigma$ such that    $ \mathrm{int}(\Sigma)\subset \mathrm{int}(W)$ and $ \partial\Sigma\subset S$. Assume that $\mathrm{int}(\Sigma)$ separates $W$ into two connected components, each having as boundary the union of $\mathrm{int}(\Sigma)$ and a domain in $S$. Fix one of these components, say $D$, and let $\Omega=\partial D\cap S$. An \emph{admissible variation} of $\Sigma$ is a variation $\{\Sigma_t\colon t\in (-\epsilon,\epsilon)\}$ such that $\mathrm{int}(\Sigma_t)\subset\mathrm{int}(W)$ and $\partial\Sigma_t\subset S$. By denoting $\Omega(t)$ to the domain bounded by $\partial\Sigma_t$ in $S$, we define the energy functional
$$
\E(t)=\mbox{Area}(\Sigma_t)-\cos\gamma\, \mbox{Area}(\Omega(t)),
$$
where $\gamma\in (0,\pi)$. A surface $\Sigma$ is a critical point of $\E$ for all volume preserving variations of $\Sigma$ if and only if the mean curvature $H$ of $\Sigma$ is constant and the angle   between $\Sigma$ and $S$ along $\partial\Sigma$ is constant and it coincides with $\gamma$. In such a case,  we say that $\Sigma$ is a {\it capillary surface}. The angle $\gamma$ is the angle formed by the unit normal vectors $N$ of $\Sigma$ and $\widetilde{N}$ of $S$ along $\partial\Sigma$, that is $\cos\gamma=\langle N,\widetilde{N}\rangle$. The vector   $N$ points into  $\Omega$ whereas $\widetilde{N}$ points outwards $\Omega$.  See \cite{rs,rv} for details. The second variation of $\E$ is 
\begin{equation*}
\E''(0)=-\int_\Sigma u\cdot \L[u]+\int_{\partial \Sigma}
u\Big(\frac{\partial u}{\partial\nu}-\q u\Big),
\end{equation*}
where
\begin{equation}\label{qq}
\q=\frac{1}{\sin\gamma} \widetilde{A}(\widetilde{\nu},\widetilde{\nu})+\frac{\cos\gamma}{\sin\gamma}\, A(\nu,\nu).
\end{equation}
Here $\widetilde{A}$ is the second fundamental  of $S$ with respect to $-\widetilde{N}$.  The vectors $\nu$ and $\widetilde{\nu}$ are the exterior  unit conormal vectors of $\partial\Sigma$ on $\Sigma$ and on $S$ respectively. Associated to the quadratic form $\E''(0)$   we also have the notions of   weak Morse index and Morse index. The eigenvalue problems \eqref{eq10} and \eqref{p-e} are the same but replacing the  condition $u=0$ in $\partial\Sigma$ by the so-called Robin condition
\begin{equation}\label{nu2}
\frac{\partial u}{\partial \nu} - \q u=0\quad \mbox{ in}\  \partial\Sigma.
\end{equation}

\section{Proof of Theorem \ref{t1}.}\label{s4}

Under the hypothesis of Thm. \ref{t1}, we know $\kappa_g^2+\kappa>0$ and $\c_\alpha$ is parametrized by \eqref{para}. By   Lem. \ref{lem1},    the Jacobi operator is 
$$\L=\Delta+\kappa_g^2+\kappa.$$
For the computation of the Laplacian $\Delta$, we use its expression in local coordinates $\psi=\psi(s,t)$, namely, 
$$
\Delta u=\frac{1}{\sqrt{\det(g_{ij})}}\sum_{i,j=1}^2\partial_i\left(\sqrt{\det (g_{ij})}\,g^{ij}\partial_j u\right),
$$
 where $\partial_1=\partial_s$ and $\partial_2=\partial_t$. Notice that $g^{11}=1$, $g^{12}=r\tau$, $g^{22}=1+r^2\tau^2$ and $\det(g_{ij})=1$.  If $u=u(s,t)$,   then it is immediate
\begin{equation}\label{eqLaplacianocoordenadas}
\Delta u=u_{ss}+2r\tau u_{st}+(1+r^2\tau^2)u_{tt}.
\end{equation}
We now compute the eigenvalues of the eigenvalue problem \eqref{p-e} for the truncated cylinders $\c_\alpha(L)$. Let us observe that the domain of $u$ is the rectangle $[0,2\pi R]\times [0,L]$ in the $(s,t)$-plane and $u$ is $2\pi R$-periodic in the $s$-variable. Consider separation of variables $u(s,t)=f(s)g(t)$, where $f=f(s)$ and $g=g(t)$ are smooth functions of one variable and $f$ is $2\pi R$-periodic. Using that $\alpha$ is a closed curve and \eqref{eqLaplacianocoordenadas}, the eigenvalue problem \eqref{p-e} is
\begin{equation}\label{eq1}
\left\{\begin{split}
&f''g+2r\tau f'g'+(1+r^2\tau^2)fg''+\left(\kappa_g^2+\kappa+\lambda\right)fg=0, \quad \mbox{in } [0,2\pi R]\times [0,L]\\
&g(0)=g(L) =0.
\end{split}\right.
\end{equation}
Here it is understood that the derivative $(\cdot)'$ is with respect to     each one of the variables of $f$ and $g$. Dividing the first equation by $fg$, we have 
$$
\frac{f''}{f}+2r\tau\frac{f'}{f}\frac{g'}{g}+(1+r^2\tau^2)\frac{g''}{g}+\kappa_g^2+\kappa+\lambda=0.
$$
Differentiating with respect to $s$ and next with respect to $t$, we obtain
$$\left(\frac{f'}{f}\right)'\left(\frac{g'}{g}\right)'=0.$$
This equation implies that either $f'/f$ or $g'/g$ are constant functions. If $g'/g=c\in\r$, then either $g(t)$ is a constant function or $g(t)=e^{ct}$. The boundary condition $g(0)=g(L)=0$ implies that if $g$ is constant then $g=0$ and hence $u(s,t)=0$, which it is not possible. In case that $g(t)=e^{ct}$,  then the boundary condition is not fulfilled. 

Consequently $f'/f=c$ for some constant $c\in\r$.   We have two possibilities.
\begin{enumerate}
\item Case $c\neq 0$. Then $f(s)=e^{cs}$ but this case is not possible because $f$ is a periodic function.  
\item Case $c=0$. Then $f(s)$ is a non-zero constant function since otherwise $u(s,t)=0$. Then Eq. \eqref{eq1} reduces to 
\begin{equation}\label{g2}
(1+r^2\tau^2)g''+\left(\kappa_g^2+\kappa+\lambda\right)g=0.
\end{equation}
The solution of this equation depends on the sign of $\kappa_g^2+\kappa+\lambda$. 
\begin{enumerate}
\item Case $\kappa_g^2+\kappa+\lambda<0$. If 
$$\delta^2=-\frac{\kappa_g^2+\kappa+\lambda}{1+r^2\tau^2},$$
then the solution of \eqref{g2} is $g(t)=A\cosh(\delta t)+B\sinh(\delta t)$, $A,B\in\r$. Imposing the boundary condition \eqref{eq1}, we arrive $A=B=0$ and $g$ would be $0$, which it is not possible. 
\item Case  $\kappa_g^2+\kappa+\lambda=0$. Then the solution of \eqref{g2} is $g(t)=At+B$, $A,B\ n\r$. Again the boundary conditions \eqref{eq1} imply $A=B=0$. This case is not possible.
\item Case $\kappa_g^2+\kappa+\lambda>0$.   Let 
$$\mu^2=\frac{\kappa_g^2+\kappa+\lambda}{1+r^2\tau^2},\quad \mu>0.$$
The solutions of \eqref{eq1} are $g(t)=A \sin(\mu t)+B\cos(\mu t)$, where $A,B\in\r$.  Since $g(0)=g(L)=0$, then $B=0$ and $\mu L\in\{n\pi:n\in\mathbb{N}\}$. Therefore   the eigenvalues $\lambda=\lambda_n$ are given by 
\begin{equation}\label{lam}
\lambda_n=\frac{n^2\pi^2}{L^2}(1+r^2\tau^2)-(\kappa_g^2+\kappa),
\end{equation}
while the eigenfunctions are $g_n(t)=\sin(\frac{n\pi}{L} t)$, $n\in\mathbb{N}$.
\end{enumerate}
\end{enumerate}
We have seen in Sect. \ref{s2} that if the second eigenvalue $\lambda_2$ is negative, then the surface is not stable. The condition $\lambda_2<0$ is fulfilled whenever $L$ satisfies
\begin{equation}\label{eqcondicionL}
L>2\pi\sqrt{\frac{1+r^2\tau^2}{\kappa_g^2+\kappa}}.
\end{equation}
This is just \eqref{pla1} after some  manipulations and we conclude the proof.

 Theorem \ref{t1} gives a sufficient criterion of instability for truncated pieces $\c_\alpha(L)$ of circular cylinders because we have restricted to find eigenfunctions $u$ that are of type $u(s,t)=f(s)g(t)$. Using separation of variables, the eigenfunctions are of type $u_n(s,t):=\sin(\frac{n\pi}{L} t)$. This implies that the variation of $\c_\alpha(L)$ associated to $u_n$ does not depend on $s$. Thus the variation of the surface is axially symmetric with respect to the axis of the cylinder, in this case, the $z$-axis of $\mathcal{R}(\kappa,\tau)$. In consequence, if we are studying stability of $\c_\alpha(L)$ only for those variation that are axially-symmetric all the above computations provide sharp estimates in the stability/instability criterion.  By simplicity in the statement, we express the critical length in terms of $r$ and $\kappa_g$ given by \eqref{lam}.

 \begin{corollary}\label{cor} Let $\c_\alpha(L)$ be a truncated piece of a circular cylinder of length $L>0$. Then $\c_\alpha(L)$ is:
 \begin{enumerate}
 \item strongly stable for axially-symmetric variations if and only if $L\leq L_0/2$. 
 \item stable for axially-symmetric variations if and only if $L\leq L_0$. 
 \end{enumerate}
 \end{corollary}

 \begin{proof}
 The statement (1) is immediate from the expression of $\lambda_1$.  For the statement (2), notice that the eigenvalues of \eqref{eq10} are those $\lambda_n$ that satisfy that the corresponding eigenfunction belongs to $\mathcal{V}$. The first eigenfunction $u_1(s,t)=\sin(\frac{\pi}{L} t)$ does not satisfy the mean zero integral, so $\lambda_1$ is not an eigenvalue of \eqref{eq10}. The next eigenvalue is $\lambda_2$ where the eigenfunction $u_2(s,t)=\sin(\frac{2\pi}{L} t)$ has mean zero integral because 
 $$\int_{\c_\alpha(L)}u_2=\int_0^{2\pi}\int_0^L \sin\left(\frac{2\pi}{L} t\right)\, dsdt=0,$$
 where we have used that $\mbox{det}(g_{ij})=1$. Thus $\lambda_2$ is the first eigenvalue of \eqref{eq10}. Then the result is immediate from the expression of $\lambda_2$.
 \end{proof}

We finish this section by a remark about the case $\kappa_g^2+\kappa\leq 0$, where $\kappa_g$ is constant. In this case, it was proved in \cite{mpr} that $\c_\alpha$ is strongly stable. This can be deduced directly from the expression of the quadratic form $Q$ defined in  \eqref{e2}. Indeed, by integration by parts, for all $u\in C_0(\c_\alpha)$  we have
$$Q[u]=\int_{\c_\alpha} |\nabla u|^2-(\kappa_g^2+\kappa)u^2\geq 0,$$
because $\kappa_g^2+\kappa\leq 0$. On the other hand, if one still wants to calculare the Morse index,   the arguments in the proof of Thm. \ref{t1} fail because now the curve $\alpha$ is not closed. However, it is possible to compute explicitly the weak index of $\c_\alpha$ and deduce that this index is $0$, proving strongly stability when $\kappa_g^2+\kappa\leq 0$.  To show an example, consider the case $\kappa<0$ and $\kappa_g=0$. Then $\alpha$ is a geodesic of the hyperbolic plane $\h^2(\kappa)$. A parametrization of the corresponding cylinder is $\c_\alpha$ is $\psi(s,t)=(s,0,t)$, with $s\in (-s_1,s_1)$, where $s_1=2/\sqrt{-\kappa}$. Since $\alpha$ is not compact, we need to take an exhaustion of $\c_\alpha$ by considering pieces of cylinders of type $\psi([-s_0,s_0]\times [-L,L])$, with $0<s_0<s_1$, and letting $s_0\to s_1$ and $L\to\infty$. Now $\psi_s=\sigma E_1$ and $\psi_t=E_3$ and the eigenvalue problem \eqref{eq1} becomes
\begin{equation*}
\left\{\begin{split}
&\frac{1}{\sigma^2}f''g+fg''+\left(\kappa+\lambda\right)fg=0, \\
&g(-L)=g(L) =0,\\
&f(-s_0)=f(s_0)=0.
\end{split}\right.
\end{equation*}
Making a similar reasoning, we obtain $g(t)=\sin(\frac{k\pi}{L}t)$, $k\in\mathbb{N}$, and 
$$\frac{1}{\sigma^2}\frac{f''}{f}+\kappa+\lambda=\frac{k^2\pi^2}{L^2},$$
or equivalently, 
$$f''+\sigma^2\left(\kappa+\lambda-\frac{k^2\pi^2}{L^2}\right)f=0.$$
The solutions of this equation depend whether the parenthesis is negative, zero or positive. By the boundary conditions $f(-s_0)=f(s_0)=0$, the first two cases are discarded and necessarily $\kappa+\lambda-\frac{k^2\pi^2}{L^2}=\delta^2$, for some $\delta>0$. Then it is immediate to deduce   $\delta=\frac{\pi n}{2s_0}$, $n\in\mathbb{N}$. In particular, we have 
$$\lambda=\lambda_{k,n}=\frac{1}{\sigma^2}\left(\delta^2+\frac{k^2\pi^2}{ L^2}-\kappa\right)>0$$
because $\kappa<0$.  Thus all compact pieces $\psi([-s_0,s_0]\times [-L,L])$ are strongly stable. If now $s_0\to s_1$ and $L\to\infty$, we deduce that $\c_\alpha$ is strongly stable.

\section{Proof of Theorem \ref{t2}}\label{s5}

We first see that $\c_\alpha$ intersects orthogonally the planes $\Pi_c$ of equation $z=c$ and, in consequence, $\c_\alpha(L)$ is a capillary surface on $\Pi_0\cup\Pi_L$ with $\gamma=\pi/2$. A parametrization of $\Pi_c$ is $\phi(x,y)=(x,y,c)$, $x,y\in\r$. Using \eqref{equi} we have    
\begin{equation}\label{n2}
\begin{split}
\phi_x&=\sigma(E_1+\tau y  E_3),\\
\phi_y&=\sigma(E_2-\tau x E_3),\\
\widetilde{N}&=\frac{1}{\sqrt{1+\tau^2(x^2+y^2)}}(-\tau y E_1+\tau x E_2+E_3).
\end{split}
\end{equation}
In the cylinder $\c_\alpha$ and thanks to the computation of $\psi_s$ and $\psi_t$ of the parametrization \eqref{para}, the unit normal $N$ of $\c_\alpha$ is  
\begin{equation*}
N=\cos\frac{s}{R} E_1+\sin \frac{s}{R} E_2.
\end{equation*}
Thus, along $\partial\Sigma$, we have 
 $\langle N,\widetilde{N}\rangle=0$. This means that the intersection is orthogonal and that the contact angle is $\gamma=\pi/2$.

 As a consequence, the function $\q$ in \eqref{qq} reduces to $\q=\widetilde{A}(\widetilde{\nu},\widetilde{\nu})$. For the computation of  $\widetilde{A}(\widetilde{\nu},\widetilde{\nu})$, we calculate the second fundamental form $\widetilde{A}$ with respect to the basis $\{\phi_x,\phi_y\}$. We know that  
 $$\widetilde{A}=\left(\begin{array}{ll}\langle\overline{\nabla}_{\phi_x}\phi_x,\widetilde{N}\rangle &\langle\overline{\nabla}_{\phi_x}\phi_y,\widetilde{N}\rangle \\
 \langle\overline{\nabla}_{\phi_y}\phi_x,\widetilde{N}\rangle &\langle\overline{\nabla}_{\phi_y}\phi_y,\widetilde{N}\rangle \end{array}\right).$$
Using \eqref{eqc}, we have 
\begin{equation*}
\begin{split}
\overline{\nabla}_{\phi_x}\phi_x&=\sigma_x E_1-(\sigma_y+2\tau^2 y\sigma^2)E_2+\sigma_x\tau y E_3,\\
\overline{\nabla}_{\phi_x}\phi_y&=(\sigma_y+\tau^2\sigma^2 y)E_1+(\sigma_x+\sigma^2\tau^2 x)E_2+\tau(-\sigma_x x-\sigma +\sigma^2)E_3 ,\\
\overline{\nabla}_{\phi_y}\phi_x&=(\sigma_y+\tau^2\sigma^2 y)E_1+(\sigma_x+\sigma^2\tau^2 x)E_2+\tau(\sigma_y
y+\sigma -\sigma^2)E_3 ,\\
\overline{\nabla}_{\phi_y}\phi_y&= -(\sigma_x+2\tau^2\sigma^2 x)E_1+\sigma_y E_2-\sigma_y\tau x E_3.
\end{split}
\end{equation*}
Thus  the matrix of $\widetilde{A}$ is 
\begin{equation}\label{aa}
\frac{\tau}{\sqrt{1+\tau^2(x^2+y^2)}}\left(\begin{array}{ll}
-x(\sigma_y+2\tau^2 y \sigma^2)& -\sigma+(1+(x^2-y^2)\tau^2)\sigma^2-y\sigma_y\\
\sigma-(1-(x^2-y^2)\tau^2)\sigma^2+x\sigma_x & y(2 x\tau^2\sigma^2+\sigma_x)\end{array}
\right).
\end{equation}
We calculate $\widetilde{\nu}$. We have
\begin{equation*}
\widetilde{\nu}=  \widetilde{N}\times\frac{\psi_s}{|\psi_s |}= N= -\frac{1}{r}(x E_1+y E_2)=-\frac{1}{r\sigma}(x\phi_x+y\phi_y).
\end{equation*}
It is now immediate from \eqref{aa} that $ \widetilde{A}(\widetilde{\nu},\widetilde{\nu})=0$ and consequently $\textbf{q}=0$. This implies that   the Robin condition \eqref{nu2} in the eigenvalue problem  is simply 
$$\frac{\partial u}{\partial\nu}=0\quad \mbox{in }\partial\c_\alpha(L).$$
The computation of $\nu$ gives 
$$\nu=N\times \frac{\psi_s}{|\psi_s |}=\frac{1}{\sqrt{1+r^2\tau^2}}(-y\tau  E_1+ x\tau E_2+E_3).$$
Since $\nabla u=f'g \psi_s+fg'\psi_t$, we have
\begin{equation}\label{nu3}
\frac{\partial u}{\partial \nu}=\langle\nabla u,\nu\rangle=\frac{fg'}{\sqrt{1+r^2\tau^2}}.
\end{equation}
Again, consider separation of variables as in Thm. \ref{t1}. Now the eigenvalue problem is \eqref{eq1} together with \eqref{nu3}, which leads to
\begin{equation}\label{eq3}
\left\{\begin{split}
&f''g+2r\tau f'g'+(1+r^2\tau^2)fg''+\left(\kappa_g^2+\kappa+\lambda\right)fg=0,\\
&g'(0)=g'(L) =0.
\end{split}\right.
\end{equation}
As in the proof of Thm. \ref{t1}, we divide by $fg$ and a similar argument allows to deduce $f'(s)=cf(s)$. Again $c=0$ because    $f(s)=e^{cs}$ is not possible by the periodicity of  $f$. Since $c=0$,  $f$ is a non-zero constant  function and     \eqref{eq3} reduces to  
\begin{equation*}
\left\{
\begin{split}
&g''+\frac{\kappa_g^2+\kappa+\lambda}{1+r^2\tau^2}g=0,\\
&g'(0)=g'(L)=0.
\end{split}
\right.
\end{equation*}

If $\kappa_g^2+\kappa+\lambda\leq 0$  then $g$ does not fulfill the boundary conditions at $t=0$ and $t=L$. Thus $\kappa_g^2+\kappa+\lambda$  must be positive. Let
$$\mu^2=\frac{\kappa_g^2+\kappa+\lambda}{1+r^2\tau^2},\quad\mu>0.$$
The boundary condition $g'(0)=0$ implies $g(t)=A\cos(\mu t)$ for some constant $A\not=0$. The other boundary condition  $g'(L)=0$ yields $\mu=n\pi/L$, $n\in\mathbb{N}$. Consequently, the eigenvalues are
$$
\lambda_n=\frac{n^2\pi^2}{L^2}(1+r^2\tau^2)-(\kappa_g^2+\kappa).
$$
These eigenvalues coincide   with the ones deduced in \eqref{lam}.   In order to achieve instability we impose $\lambda_2<0$ which yields
$$
L>2\pi\sqrt{\frac{1+r^2\tau^2}{\kappa_g^2+\kappa}}.
$$
This is again \eqref{eqcondicionL} and we conclude the proof.

Theorem \ref{t2} considers the stability in the partitioning problem of circular cylinders. We now treat vertical cmc cylinders when $\kappa_g^2+\kappa\leq 0$ and its stability in the partitioning problem. Notice that by Lem. \ref{lem1}, the mean curvature of $\c_\alpha$ is constant but this is not sufficient to be a capillary surface because of the condition of contact angle. In fact,  the intersection of $\c_\alpha$ with the planes $\Pi_c$ is not orthogonal unless that $\tau=0$. The condition $\kappa_g^2+\kappa\leq 0$ implies that the curve $\alpha$ is not closed.  We have two cases.
\begin{enumerate}
\item Case $\kappa_g=0$, that is, $\alpha$ is a geodesic. Without loss of generality, we can assume that $\alpha$ is parametrized by $\alpha(s)=(s,0,0)$ and $\c_\alpha$ by $\psi(s,t)=(s,0,t)$. Then $\psi_s=\sigma E_1$ and $\psi_t=E_3$. This gives $N=E_2$. On the other hand, from \eqref{n2} we have   
$$\widetilde{N}=\frac{1}{\sqrt{1+s^2\tau^2}}(s\tau E_2+E_3)$$
along $\c_\alpha\cap\Pi_c$. This gives $\langle N,\widetilde{N}\rangle=s\tau/\sqrt{1+s^2\tau^2}$ which it is only constant if $\tau=0$. This case  corresponds when  $\ekt$ is the product space $\h^2(\kappa)\times\r$. 

\item Case $\kappa_g\not=0$. Then necessarily $\kappa<0$. Now $\alpha$ is a horocycle or an equidistant curve. The parametrization of $\alpha$ is 
$\alpha(s)=(r\cos\frac{s}{R},r\sin\frac{s}{R}-y_0)$ where $y_0\in (0,\frac{2}{\sqrt{-\kappa}}]$ and $r>\frac{2}{\sqrt{-\kappa}}-y_0$. Since the parametrization of  $\c_\alpha$ is $\psi(s,t)=(\alpha(s),t)$, using \eqref{equi} we have 
\begin{equation*} 
\begin{split}
\psi_s&=\frac{r\sigma}{R}\left(-\sin\frac{s}{R} E_1+\cos \frac{s}{R} E_2+(r\tau+y_0\sin\frac{s}{R})E_3\right), \\
\psi_t&=E_3.
\end{split}
\end{equation*}
Thus
$$N=\cos\frac{s}{R} E_1+\sin \frac{s}{R} E_2.$$
Now along $\c_\alpha\cap\Pi_c$, we have
$$\langle N,\widetilde{N}\rangle=\frac{y_0\tau\cos \frac{s}{R}}{\sqrt{1+\tau^2(r^2+y_0^2-2ry_0\sin \frac{s}{R})}}.$$
As a consequence,  the contact angle along  the intersection curve is constant if and only if $\tau=0$ and, in such a case,   $\gamma=\pi/2$.  

\end{enumerate}
Summarizing, the case $\kappa_g^2+\kappa\leq 0$ only occurs if $\ekt$ is $\h^2(\kappa)\times\r$. In the following result, we prove that $\c_\alpha$ is strongly stable. This extends to the partitioning problem the analogous situation of the case proved in \cite{mpr}.

\begin{theorem} \label{t5}
Let $\c_\alpha$ be a vertical cmc cylinder in $\h^2(\kappa)\times\r$ such that $\kappa_g^2+\kappa<0$. Then, for any $L>0$ the truncated cylinder $\c_\alpha(L)$ is strongly stable  on $\Pi_0\cup\Pi_L$ in the partitioning problem.
\end{theorem}

\begin{proof}
The proof is a direct computation of the quadratic form $\E_p''(0)$. Let $u\in C_0(\c_\alpha(L))$. By an integration by parts, we have
$$\E_p''(0)= -\int_{\c_\alpha(L)} u\left(\Delta u+(\kappa_g^2+\kappa)u\right) +\int_{\partial \c_\alpha(L)}
u \frac{\partial u}{\partial\nu}  ds=\int_{\c_\alpha(L)} |\nabla u|^2-(\kappa_g^2+\kappa)u^2\geq 0$$
because $\kappa_g^2+\kappa\leq 0$.
\end{proof}

\section*{Acknowledgements} Antonio Bueno has been partially supported by CARM, Programa Regional de Fomento de la Investigaci\'on, Fundaci\'on S\'eneca-Agencia de Ciencia y Tecnolog\'{\i}a Regi\'on de Murcia, reference 21937/PI/22. Rafael L\'opez  is a member of the IMAG and of the Research Group ``Problemas variacionales en geometr\'{\i}a'',  Junta de Andaluc\'{\i}a (FQM 325). This research has been partially supported by MINECO/MICINN/FEDER grant no. PID2020-117868GB-I00,  and by the ``Mar\'{\i}a de Maeztu'' Excellence Unit IMAG, reference CEX2020-001105- M, funded by MCINN/ AEI/10.13039/501100011033/ CEX2020-001105-M.


\end{document}